\documentclass{article}
\usepackage[paper=a4paper,margin=1in]{geometry}

\usepackage[numbers,sort]{natbib}
\bibpunct[, ]{[}{]}{,}{n}{}{,}%

\usepackage{amsfonts}
\usepackage{amssymb}
\usepackage{url,graphicx,tabularx,array,geometry}
\usepackage{amsmath,lipsum}
\usepackage{color}
\usepackage{epstopdf}
\usepackage{verbatim}
\usepackage{blkarray} 
\usepackage{enumitem} 
\usepackage{multirow} 
\usepackage{blkarray}
\usepackage{float}
\usepackage{hyperref}
\usepackage{subfig}
\usepackage{soul} 

\usepackage{algorithm}
\usepackage[noend]{algpseudocode}

\usepackage{geometry}

\usepackage{amsthm} 
\newtheorem{theorem}{Theorem}[section]

\newtheorem{proposition}[theorem]{Proposition}
\newtheorem{corollary}[theorem]{Corollary}

\theoremstyle{remark}
\newtheorem{remark}{Remark}

\begin{document}

\title{The quadratic minimum spanning tree problem: lower bounds via extended formulations}
\author{ {Renata Sotirov}  \thanks{Department of Econometrics and OR, Tilburg University, The Netherlands, {\tt r.sotirov@uvt.nl}}
	\and
Zoe Verch\'ere
\thanks{Doctorant \'a UMA, ENSTA Paris, Institut Polytechnique de Paris, 91120 Palaiseau, France, {\tt zoe.verchere@ensta-paris.fr}}
 }

\date{}

\maketitle

\begin{abstract}
The quadratic minimum spanning tree problem (QMSTP) is the problem of finding a spanning tree of a graph such that   the total interaction  cost between pairs of edges
in the tree is minimized.
We first show  that the bounding approaches  for the QMSTP in the literature are closely related.
Then, we exploit an extended formulation for the minimum spanning tree problem to
derive a sequence of relaxations for the QMSTP with increasing complexity and quality.
The resulting relaxations differ from the relaxations in the  literature.
Namely, our  relaxations have a polynomial number of constraints and can be efficiently solved by a cutting plane algorithm.
Moreover our bounds outperform most of the bounds from the literature.

\end{abstract}


\section{Introduction}
The quadratic minimum spanning tree problem (QMSTP) is the problem of finding a spanning tree of a connected and undirected graph
such that the sum of interaction costs over all pairs of edges in the tree is minimized.
The QMSTP  was  introduced by Assad and Xu \cite{AssadXu} who have proven that the problem is strongly ${\mathcal NP}$-hard.
The QMSTP remains ${\mathcal NP}$-hard   even  when the cost matrix is of rank one \cite{Punnen2001}.
The adjacent-only  quadratic minimum spanning tree problem (AQMSTP),
that is the QMSTP where the interaction costs  of all non-adjacent edge pairs are assumed to be zero, is also introduced in \cite{AssadXu}.
That special version of the QMSTP is also  strongly  ${\mathcal NP}$-hard.
The QMSTP has applications in fields such as telecommunication, irrigation, transportation, and energy and hydraulic
networks, see e.g., \cite{AssadXu,CHIANG2007734,Chou1973AUA}.

The QMSTP can be seen as a generalization of several well known optimization problems, including the quadratic assignment problem \cite{AssadXu}
and the satisfiability problem \cite{Cordone2012SolvingTQ}.
There are also numerous variants of the QMSTP problem such as the minimum spanning tree problem with conflict pairs,
 the quadratic bottleneck spanning tree problem, and the bottleneck spanning tree problem with conflict pairs.
 For description of those problems, see e.g.,  \citet{CUSTIC201873}.
 In the same paper the authors investigate the complexity of the QMSTP and its variants, and prove intractability results for the mentioned variants
 on fan-stars, fans, wheels, $(k, n)$--accordions, ladders and $(k, n)$--ladders.
 The authors also prove that the AQMSTP on $(k, n)$--ladders and the  QMSTP where the cost matrix  is a permuted doubly graded matrix  are polynomially solvable.

\citet{Custic2015ACO} prove that the QMSTP on a complete graph  is linearizable if and only if a symmetric cost matrix  is a symmetric weak sum matrix.
An instance of the QMSTP  is called linearizable  if there exists an instance of the minimum spanning tree problem  (MSTP) such that the associated
costs for both problems are equal for every feasible spanning tree.
 The authors of \cite{Custic2015ACO}  also  present linearizability results for  the  QMSTP on complete bipartite graphs,
 on the class of graphs in which every biconnected component is either
a clique, a biclique or a cycle, as well as on biconnected graphs that contain a vertex with degree two.
Note that linearizable instances for the QMSTP can be solved in polynomial time.

There is a substantial amount of research on lower bounding approaches and exact algorithms for the QMSTP.
\citet{AssadXu} propose a lower bounding procedure that
generates a sequence of lower bounds. In each iteration of the algorithm, they solve a  minimum spanning tree problem.
\citet{ONCAN20101762} propose a Lagrangian relaxation scheme based on the linearized QMSTP formulation from \citep{AssadXu}
 that is strengthened by adding two sets of  valid inequalities.
\citet{Cordone2012SolvingTQ} re-implement the Lagrangian approach from  \citep{AssadXu} with some minor modifications that result
in improved CPU times but slightly weaker lower bounds.
Lower bounding approaches in \cite{ONCAN20101762,AssadXu,Cordone2012SolvingTQ} yield  Gilmore-Lawler (GL) type  bounds for the QMSTP.
The GL procedure is a well-known approach to construct lower bounds for quadratic binary optimization problems,
see e.g., \cite{Gilmore,Lawler}.
\citet{{PEREIRA2015149}} derive lower bounds for the QMSTP  using the reformulation linearization technique (RLT) and
introduce bounding procedures based on Lagrangian relaxation.
 The RLT provides a hierarchy of relaxations that are ranging between the continuous and convex hull representation for linear $0-1$ mixed-integer programs \cite{SheraliAdams90,SHERALI199483}.
In~\cite{sherali2013reformulation}, the RLT is generalized for solving discrete and continuous nonconvex problems.
The RLT consists of the following two steps; a reformulation step in which additional nonlinear valid inequalities are generated,
and a linearization step in which each product term  is replaced by a continuous variable. The level of the  hierarchy corresponds to the degree of polynomial terms produced during the reformulation stage.
The authors from  \cite{PEREIRA2015149}  report solving instances with up to 50 vertices to optimality.
\citet{Rostami2015LowerBF} consider  the GL procedure, as well as the (incomplete) first and second level RLT bounds for the QMSTP.
To compute those bounds, the authors solve  Lagrangian relaxations.
Relaxations presented in \cite{PEREIRA2015149,Rostami2015LowerBF} are not complete first and second level RLT relaxations.
Namely, those relaxations do not contain  a large class of RLT constraints.
\citet{GUIMARAES202046} consider semidefinite programming (SDP) lower bounds for the QMSTP and report the best exact solution approach for the problem up to date.
The SDP bounds do not belong to the RLT type bounds.
Exact approaches for the QMSTP are also considered in \cite{AssadXu,Cordone2012SolvingTQ,PEREIRA2015149,Pereira2015BranchandcutAB}.

Various heuristic approaches for the QMSTP are tested in the literature.
For example, genetic algorithms for the QMSTP are implemented in \cite{ZHOUT1998229,CHIANG2007734,GAO2005773},
a tabu search in \cite{LozanoEtAl,Cordone2012SolvingTQ,{PalubeckisEtAl}},  a swarm intelligence approach in \cite{SUNDAR20103182}, and evolutionary algorithms  in \cite{SoakCorne,Soak2006TheER}.
 \citet{PalubeckisEtAl} compare simulated annealing, hybrid genetic and iterated tabu search algorithms for the QMSTP.
The results show that their tabu search algorithm outperforms the other two approaches in terms of both solution quality and computation time.
A local search algorithm for the QMSTP that alternatively performs  a neighbourhood search and a random move phase is introduced in \cite{ONCAN20101762}.

\medskip\medskip

\subsection{Main results and outline}
Most of the lower bounding approaches for the QMSTP are closely related as we show in the next section.
Namely, those bounds are based on Lagrangian relaxations obtained from the RLT type of bounds
and solved by specialized iterative methods.
This connection was not made till now, to the best of our knowledge.
Even  semidefinite programming lower bounds  for the QMSTP involve Lagrangian relaxations.

In this paper we use a different approach  for solving the QMSTP.
We exploit an extended formulation for the minimum spanning tree problem to compute lower bounds for the QMSTP.
Our relaxations have a polynomial number of constraints and  we solve them using a cutting plane algorithm.

To derive our relaxations we first consider a linearization based relaxation for the QMSTP.
The linearization based relaxation finds the best possible linearizable matrix  for the given QMSTP  instance, by solving a linear programming (LP) problem.
In particular, it searches for the best under-estimator of the quadratic objective function that is in the form of a  weak sum matrix.
Linearization based bounds are introduced by \citet{HuSotirov2} and further exploited by \citet{MeijerSotirov}.
Next, we consider the continuous relaxation of an exact linear formulation for the QMSTP from  \cite{AssadXu}, and prove that its dual
 is equivalent to the  linearization based relaxation. To derive both relaxations, we exploit the polynomial size extended formulation for the MSTP by  \citet{MARTIN1991119}.

In order to improve the continuous relaxation of the QMSTP formulation from  \cite{AssadXu}
we add facet defining inequalities of the Boolean Quadric Polytope (BQP) \cite{Padberg}.
We show that the linearization based relaxation with a  particular subset of the  BQP inequalities
is not dominated by the incomplete first level RLT relaxation from the literature  \cite{PEREIRA2015149,Rostami2015LowerBF}, or vice versa.
However, after adding all BQP cuts to the linearization based relaxation  the resulting bounds  outperform even the incomplete
second level RLT bounds from the literature   on some instances.
Since adding all BQP inequalities at once is computationally expensive, we  implement a cutting plane approach that considers the most violated constraints.

This paper is organized as follows.
In Section \ref{TheQMST} we formally  introduce the  QMSTP and present the extended formulation for the MSTP from \cite{MAGNANTI1995503}.
In Section \ref{sect:overwierBNDS} we provide an overview of lower bounds for the QMSTP.
In particular, Section \ref{GilmoreLawler} presents the Gilmore-Lawler type bounds including Assad-Xu and \"{O}ncan-Punnen bounds,
and Section \ref{sect:RLTold} the RLT type bounds. New relaxations are presented in Section \ref{sect:newBounds}.
Section \ref{sect:linBasedBnd} introduces a linearization based relaxation, and Section \ref{sec:newLower} provides several new relaxations of increasing complexity.
Numerical results and   concluding remarks are in Section \ref{sect:nnumerics} and Section \ref{conclusion}, respectively. \\

\subsection*{Notation}
Given a subset $S \subseteq V$ of vertices in a graph $G=(V,E)$,  we denote the set of edges with both endpoints in $S$
by $E(S):=\{e \in E~|~e=\{i,j\},~  i,j \in S \}$.
Further, we use $\delta(S) \subseteq E$ to denote the set of edges with exactly one endpoint in $S$.

We introduce the operator $\text{Diag}: \mathbb{R}^{n} \rightarrow \mathbb{R}^{n \times n}$ that maps a vector to a diagonal matrix
whose diagonal elements correspond to the elements of the input vector.
For two matrices $X=(x_{ij}),~Y=(y_{ij}) \in \mathbb{R}^{n\times m}$,  $X\geq Y $ means $x_{ij} \geq y_{ij}$ for all $i,j$.
We denote by ${\mathbf 0}$ the  all-zero vector of appropriate size.

\section{The Quadratic Minimum Spanning Tree Problem} \label{TheQMST}

Let us formally introduce the quadratic minimum spanning tree problem.
We  are given a connected,  undirected graph $G=(V,E)$ with $n=|V|$ vertices and $m=|E|$ edges,
and a matrix $ Q = (q_{ef}) \in \mathbb{R}^{m\times m} $ of interaction costs between edges of $G$.
For $e=f$, $q_{ee}$ represents the cost of edge $e$.
We assume without loss of generality that $Q=Q^{\mathrm T}$.
The QMSTP can be formulated as follows:
\begin{equation} \label{eq:QMST}
\min \left \{ \sum_{e \in E} \sum_{f \in E} q_{ef} x_e x_f \, \, | ~~ x \in \mathcal{T} \right \},
\end{equation}
where $\mathcal{T}$ denotes the set of all spanning trees, and each spanning tree is represented by an incidence vector $x$ of length $m$, i.e.,
\begin{align}\label{T}
\mathcal{T} := \left \{ x \in \{0,1\}^{m} \, \, | ~~\sum_{e\in E} x_e = n-1, ~ \sum_{e \in E(S)} x_e \leq |S| -1, ~S\subset V, ~ |S|\geq 2 \right \}.
\end{align}
We denote by $\mathcal{T}_{\mathcal R}$ the convex hull of the elements from $\mathcal{T}$.

If the matrix $Q$ is a diagonal matrix i.e., $Q={\rm Diag}(p)$ for some $p\in \mathbb{R}^m$, then the QMSTP reduces to the minimum spanning tree problem:
\begin{equation} \label{eq:MSTP}
\min \left \{  \sum_{e\in E} p_e x_e  \, \, | ~~ x \in \mathcal{T} \right \}.
\end{equation}
It is a well known result by   \citet{Edmonds1971MatroidsAT} that the linear programming  relaxation of \eqref{eq:MSTP} has an integer polyhedron.
However, the spanning tree polytope has an exponential number of constraints.
Nevertheless, the MSTP is solvable in polynomial time by e.g., algorithms developed by  \citet{Prim1957} and \citet{Kruskal1956}.

There exists a polynomial size extended formulation for the minimum spanning tree problem, due to \citet{MARTIN1991119}.
He derived the polynomial size formulation for the MSTP by exploiting
a known result (see e.g., \cite{MAGNANTI1995503}) that one can test  if a vector violates subtour elimination constraints by solving a max flow problem.
 \citet{MARTIN1991119} uses the following valid separation  linear program
 for subtour elimination constraints  that contain vertex $k$:
\begin{equation} \label{separation1}
(SP_k(x)) \quad
    \begin{array}{rll}
    \text{max} & \hspace{1mm}\sum\limits_{e \in E}x_e \alpha^k_e -
     \sum\limits_{i=1, i\neq k}^n \theta^k_i\\[2.5ex]
        \text{s.t.} & \hspace{1mm} \alpha^k_e - \theta^k_i \leqslant 0 &  \forall e \in \delta(i), ~ i \in V  \\[1ex]
    & \hspace{1mm}\theta^k_i \geqslant 0 & \forall i.
    \end{array}
\end{equation}
The above max flow problem formulation is from \cite{Rhys}.
Thus, for $\tilde{x} \in \mathbb{R}^{m}$, $\tilde{x} \geq 0$, the objective value of  $SP_k(\tilde{x})$ equals zero if and only if
$ \sum_{e \in E(S)}\tilde{x}_e \leq |S| -1$ where  $k\in S$, see \cite{MARTIN1991119}.
By exploiting the dual problem of \eqref{separation1} for all vertices $k \in V,$ \citet{MARTIN1991119}
derived  the following extended formulation for the minimum spanning tree problem:
\begin{subequations}
\begin{align}
\min & ~~\sum_{e \in E} p_e x_e  \label{eq:eq1}  \\
{\rm s.t.} & ~~\sum_{e \in E} x_e = n-1  \label{eq:con1}  \\
     & z_{kij} + z_{kji} = x_e  & k = 1,\dots, n,  ~e=\{i,j\} \in E \label{eq:lincon2} \\
   & \sum_{s>i} z_{kis} + \sum_{h<i} z_{kih} \leq 1  & k = 1,\dots,n, ~i \neq k \label{eq:con3}\\
    & \sum_{s>k} z_{kks} + \sum_{h<k} z_{kkh} \leq 0  & k = 1,\dots,n \label{eq:con4} \\
      & z_{kij}\geq 0  & k,i,j = 1,\dots,n \label{eq:con5}  \\
    & x_e\geq 0 &\forall e \in E. \label{eq:conLast}
\end{align}  \label{extendModel}
\end{subequations}
There are $O(n^3)$ constraints and  $O(n^3)$  variables in the above LP relaxation.
Although it is impractical to use \eqref{extendModel} to find the minimal cost spanning tree because there exist
 several efficient algorithms for solving the MSTP,
we show that it is beneficial to use an extended formulation for solving the QMSTP.
There are several  equivalent formulations  of the extended model \eqref{extendModel}, see e.g., \cite{Kaibel}.

Let us define
\begin{equation} \label{TE}
\mathcal{T}_{E} := \left \{x \in \mathbb{R}^m,~ z \in \mathbb{R}^{m \times m \times m}|~  \eqref{eq:con1}-\eqref{eq:conLast}    \right \}.
\end{equation}
Now, the QMSTP can be formulated as follows:
\begin{equation} \label{eq:QMST1}
\min \left \{ \sum_{e \in E} \sum_{f \in E} q_{ef} x_e x_f \, \, | ~~ (x, z) \in \mathcal{T}_E \right \}.
\end{equation}
For future reference, we present  the dual of \eqref{extendModel}:
\begin{subequations}
\begin{align}
\max & ~~{-(n-1)\epsilon - \sum_{i=1}^n \sum_{{j=1,j\neq i}}^n \mu_{ij}} & \label{objDual1}\\
 {\rm s.t.} & ~~  \sum_{k=1}^n \theta_{ke} - \epsilon ~\leq p_e & \forall e \in E \\
 & ~~ \mu_{ki} + \sum_{e \in E(\{i,j\})} \theta_{ke} ~\geq 0 & k,i,j = 1,\dots,n \\
& ~~ \mu_{ki} \geq 0 & k,i = 1,\dots,n \\
& ~~ \epsilon\in \mathbb{R}, ~~ \theta_{ke} \in \mathbb{R}     & k= 1,\dots,n, ~e\in E.
\end{align}  \label{DualextendModel}
\end{subequations}

\section{An overview of lower bounds for the QMSTP} \label{sect:overwierBNDS}

In this section we present several known lower bounding approaches for the QMSTP.
We classify bounds into the Gilmore-Lawler type bounds and  the RLT type bounds.
Then, we show how those two types of bounding approaches are closely related.

\subsection{The Gilmore-Lawler type bounds} \label{GilmoreLawler}

We consider here the  classical Gilmore-Lawler  type bound for the QMSTP.
The GL procedure is a well-known approach to construct lower bounds for binary quadratic optimization problems.
The bounding procedure was introduced by \citet{Gilmore} and \citet{Lawler} to compute a lower bound for the quadratic assignment problem.
Nowadays, the  GL bounding procedure is extended to many other optimization  problems including  the quadratic minimum spanning tree problem \cite{Rostami2015LowerBF},
the quadratic shortest path problem \cite{Rostami}, and the quadratic cycle cover problem \cite{MeijerSotirov}.

The GL procedure  for the QMSTP is as follows.
For each edge $e \in E$, solve the following optimization problem:
\begin{equation}\label{ze}
z_e :=\min \left \{ \sum\limits_{f \in E} q_{ef} x_f \, \, | ~~ x_e = 1, ~x \in \mathcal{T} \right \}.
\end{equation}
The value $z_e$ is the minimum contribution to the QMSTP objective with $e$ being in the solution.
Then, solve the following minimization problem:
\begin{equation}  \label{GLb}
GL(Q) := \min \left \{ \sum\limits_{e \in E} z_e x_e \, \, | ~~ x \in \mathcal{T} \right \}.
\end{equation}
The optimal solution of the above optimization  problem is the  GL type lower bound for the QMSTP.
Note that in  \eqref{ze} and \eqref{GLb} one can equivalently solve the corresponding continuous relaxations,
thus optimize over $\mathcal{T}_{\mathcal R}$ or  $\mathcal{T}_E$, or use one of the efficient algorithms for solving the MSTP.

\citet{Rostami2015LowerBF} prove that one can also compute the GL type  bound by solving the following LP problem:
\begin{subequations}
\begin{align}
GL(Q) := \min & \sum_{e,f \in E} q_{ef}y_{ef} \\
{\rm s.t.} & ~~\sum_{f \in E} y_{ef} = (n-1) x_e \qquad   \forall  e \in E \label{con GLBone}\\
 & \sum_{f \in E(S)} y_{ef} \leq (|S|-1)x_e \quad  ~ \forall  e \in E, ~\forall S \subset V, ~S \neq \emptyset \label{con leqy} \\
 & y_{ee} = x_e \quad \forall e \in E \label{yeexee}\\
&  y_{ef} \geq 0 \quad \forall e,f \in E   \\
&  x \in \mathcal{T}_{\mathcal R}. \label{con GLBlast}
\end{align}  \label{MILP_GL}
\end{subequations}
Note that from \eqref{con leqy} it follows that $y_{ef} \leq x_e$ for all $e,f\in E$.
Clearly, it is more efficient to compute the GL bound by solving \eqref{ze}--\eqref{GLb}, than solving \eqref{MILP_GL}.

The Gilmore-Lawler  type bound can be further improved within an iterative algorithm that exploits equivalent reformulations of the problem.
The so obtained iterative bounding scheme is known as the generalized Gilmore-Lawler bounding scheme.
The generalized GL bounding scheme was implemented for several optimization problems, see e.g., \cite{Carraresi,MeijerSotirov,Rostami,HuSotirov2}.
\citet{HuSotirov2} show that bounds obtained by the GL bounding  scheme are dominated by the first level RLT bound.
The bounding procedure in the next section can be seen as a special case of the generalized Gilmore-Lawler bounding scheme.

\subsubsection{Assad-Xu leveling procedure}

\citet{AssadXu} propose a method that generates a monotonic sequence of lower bounds for the QMSTP,
which results in an improved GL type lower bound.
The approach is based on equivalent reformulations of the QMSTP.

 \citet{AssadXu} introduce first the following transformation of costs:
\begin{equation} \label{gefGamma}
\begin{array}{ll}
        q_{ef}(\gamma) = q_{ef} + \gamma_f, &  \forall e,f \in E,~ e\neq f  \\[1ex]
        q_{e,e} (\gamma) = q_{e,e} - (n-2) \gamma_e, & \forall  e\in E,
    \end{array}
\end{equation}
where $\gamma_e$ ($e\in E$) is a given parameter, and then  apply the GL procedure for a given $\gamma$.
In particular, \citet{AssadXu} solve the following problem:
\begin{equation}\label{fgamma}
 f_e(\gamma) :=  \min  \left \{  \sum\limits_{f \in E} q_{ef}(\gamma)  x_f  \, \, | ~~  x_e = 1, ~x \in \mathcal{T} \right \},
\end{equation}
for each edge $e \in E$.
The value $f_e(\gamma)$ is the minimum contribution to the reformulated QMSTP objective with  $e$ being in the solution.
Then $f_e(\gamma)$  is used in the  following optimization problem:
\begin{equation} \label{AX}
 AX(\gamma) :=  \min \left \{ \sum\limits_{e \in E} f_e(\gamma)  x_e  \, \, | ~~   x \in \mathcal{T} \right \}.
\end{equation}
The optimal solution of the above optimization problem is a GL type lower bound for the QMSTP that depends on a parameter $\gamma$.
In the case that $\gamma$ is all zero-vector, the corresponding lower bound $AX(\gamma)$ is
equal to the GL type lower bound described in the previous section.

The function $AX(\gamma)$ is a piecewise linear, concave function.
In order to find $\gamma$ that provides the best possible bound of type \eqref{AX}, that is to solve $\max_{\gamma} AX(\gamma)$,
the authors propose Algorithm \ref{AlgorithmAX}.

\begin{algorithm}[H]
\small
\caption{Assad-Xu leveling procedure}\label{AlgorithmAX}
\begin{algorithmic}[1]
\State $\epsilon >0$, $\gamma^1 = {\mathbf 0}$, $i \gets  1$
\While {$(\max\limits_{e} f_e(\gamma^i) - \min\limits_{e} f_e(\gamma^i) > \epsilon$)}
\State Compute transformed costs using \eqref{gefGamma}.
\State Compute $ f_e(\gamma^i)$ and   $AX(\gamma^i) $ using  \eqref{fgamma} and \eqref{AX}, resp.
\State Update $\gamma^{i+1}_e \longleftarrow  \gamma_e^i + \frac{f_e(\gamma^i)}{n-1} $ for $e\in E$.
\State $i \gets i + 1$
\EndWhile
\State
\Return $AX(Q) \leftarrow AX(\gamma^{i-1})$.
\end{algorithmic}
\end{algorithm}
Algorithm   \ref{AlgorithmAX} resembles  the generalized Gilmore-Lawler bounding scheme.
Since the algorithm starts with $\gamma^1 ={\mathbf 0}$, the first computed bound equals the classical GL type bound.
Note that Algorithm \ref{AlgorithmAX}, in Step 3  adds to the off diagonal elements of the cost matrix $Q$ a weak sum matrix and then subtracts its linearization vector on the diagonal.
For a relation between linearizable  weak sum matrices and corresponding linearization vectors see Section \ref{sect:linBasedBnd}.
Thus Algorithm   \ref{AlgorithmAX}  generates a sequence of equivalent representations of the QMSTP while converging to the optimal $\gamma^*$.

\subsubsection{ \"{O}ncan-Punnen  bound}

The following exact, linear formulation for the QMSTP is presented in \cite{AssadXu}:
\begin{subequations}\label{AXexact}
\begin{align}
 \min & \sum_{e,f \in E} q_{ef}y_{ef} & \label{objOP} \\
{\rm s.t. } &	~~\sum_{f \in E} y_{ef} = (n-1) x_e & \forall e \in E \label{firstOP} \\
& \sum_{e \in E} y_{ef} = (n-1)x_f  & \forall f \in E \\
&  y_{ee} = x_e  & \forall e \in E \\
&  0 \leq y_{ef} \leq 1 & \forall e,f \in E \\
&  x \in \mathcal{T}. \label{lastOP}
\end{align}
\end{subequations}
\citet{ONCAN20101762} propose
 a lower bound for the QMSTP that is derived from  \eqref{AXexact}
 and also includes the following  valid inequalities:
\begin{eqnarray}
\sum_{e \in \delta(i)}  y_{ef} \geq x_f  &  \forall  i \in V, f \in E  \label{eq:EFVI1} \\
\sum_{f \in \delta(i)}  y_{ef} \geq x_e, &  \forall i \in V, e \in E, \label{eq:EFVI2}
\end{eqnarray}
where $\delta(i)$ denotes the set of all incident edges to vertex $i$.
In particular, the   authors from \cite{ONCAN20101762}  propose to solve the Lagrangian relaxation  of  \eqref{objOP}--\eqref{lastOP}, \eqref{eq:EFVI1}--\eqref{eq:EFVI2}
 obtained by dualizing constraints (\ref{eq:EFVI1}).  For a fixed Lagrange mulitiplier $\lambda$,
  the resulting Lagrangian relaxation is solved in a similar fashion as the GL type bound,
where costs in objectives of problems  \eqref{ze} and  \eqref{GLb} are:
\[
{q}_{ef}(\lambda) := q_{ef}- \sum_{i\in V, ~f \in \delta(i)} \lambda_{i,f} \qquad \forall e,f \in E, ~e\neq f
\]
and
\[
{q}_{ee}(\lambda) := q_{ee} + \sum_{i\in V, ~e \in \delta(i)}\lambda_{i,e} \qquad  \forall  e \in E,
\]
respectively.
After the solution of the Lagrangian relaxation with a given $\lambda$ is obtained, a  subgradient algorithm is implemented to update the  multiplier, and the process iterates.
Numerical results show that \"{O}ncan-Punnen  bounds provide, in general, stronger bounds than Assad-Xu bounds, see also \cite{Rostami2015LowerBF}.
In  \cite{ONCAN20101762}, the authors  also consider  replacing the condition $x \in \mathcal{T}$ by the multicommodity flow constraints from \cite{MAGNANTI1995503}, but
only for graphs up to 20 vertices.

\subsection{The  RLT type bounds} \label{sect:RLTold}

\citet{PEREIRA2015149} present a relaxation for the QMSTP that is  derived by applying the first level  reformulation linearization technique on \eqref{eq:QMST}.
Nevertheless, their numerical results are obtained by solving the following {\em incomplete}  first level RLT relaxation:
\begin{equation}\label{RLTu}
\begin{array}{rl}
 \widetilde{RTL_1}(Q):=\min & \sum\limits_{e,f \in E} q_{ef}y_{ef} \\[1.5ex]
{\rm s.t.} & y_{ef}=y_{fe} \qquad \forall  e,f \in E  \\
& \eqref{con GLBone} - \eqref{con GLBlast}.\\
\end{array}
\end{equation}
We denote by $\widetilde{RLT_1}$ the  relaxation \eqref{RLTu}.
To complete the above model to  the first level RLT formulation of the QMSTP, one needs to add the following constraints:
\begin{equation}\label{const:extra}
 \sum_{f \in E(S)}(x_f- y_{ef}) \leq (|S|-1)(1-x_e)  \qquad    \forall e \in E,~  \forall S \subset V, ~|S| \geq 2.
\end{equation}
It is written in  \cite{PEREIRA2015149}  that preliminary computational experiments show that  bounds obtained from the
(complete) first level RLT relaxation ($RLT_1$) do not significantly differ from  bounds obtained from the incomplete first level RLT relaxation.
However, computational effort to solve the full RLT relaxation significantly increases due to the constraints \eqref{const:extra}.
Therefore,  \citet{PEREIRA2015149}
compute only the incomplete first level RLT bound.

The (complete) first level RLT formulation for the QMSTP needs   not to include constraints:
\begin{eqnarray}
y_{ef} \leq x_e, \quad  y_{fe} \leq x_e, &  \forall e,f \in E, \label{econ:extra1} \\
y_{ef} \geq x_e+x_f -1, &  \forall e,f \in E, \label{econ:extra2}
\end{eqnarray}
since these are implied by the rest of the constraints. Namely, constraints \eqref{econ:extra1}--\eqref{econ:extra2}
readily follow by considering $S$ with two elements in \eqref{con leqy} and \eqref{const:extra}.
However, constraints \eqref{econ:extra2} are not implied by the constraints of the incomplete RLT relaxation  \eqref{RLTu}.
Nevertheless, \citet{PEREIRA2015149}  do not impose those constraints to the incomplete model for the same reason that constraints \eqref{const:extra} are not added;
that is that constraints \eqref{econ:extra2} do not significantly improve the value of the bound but are expensive to include.

To approximately solve  \eqref{RLTu},  \citet{PEREIRA2015149} dualize constraints $y_{ef}=y_{fe}$ ($ e,f \in E$) for a given Lagrange multiplier
and then apply the GL procedure. Then, they use a subgradient algorithm to derive a sequence of improved  multipliers and compute the corresponding bounds.
 The so obtained bounds converge to the optimal value of the incomplete first level RLT lower bound.
It is clear from the above discussion that the bound from \cite{PEREIRA2015149} is  related to the Gilmore-Lawler type bounds.

\citet{PEREIRA2015149} prove
\[
 P_{\widetilde{RLT_1}} \subseteq P_{OP} \subseteq P_{AX},
\]
where $P_{AX}$ denotes the convex hull of the   QMSTP formulation \eqref{AXexact},
$P_{OP}$ denotes the convex hull of \eqref{firstOP}--\eqref{lastOP} and \eqref{eq:EFVI1}--\eqref{eq:EFVI2},
and $P_{\widetilde{RLT_1}}$ the convex hull of the incomplete first level RLT relaxation   \eqref{RLTu}.
We remark that  in  \cite{PEREIRA2015149} the authors add to $P_{AX}$ and $P_{OP}$ constraints \eqref{con leqy} that are not in the
original description of those polyhedrons, see \cite{ONCAN20101762}.
Nevertheless, their result as well as the statement above are correct.

\begin{remark}
It is interesting to note that the continuous relaxation \eqref{MILP_GL} and the incomplete first level RLT relaxation differ only in
the symmetry constraints $y_{ef}=y_{fe}$ ($\forall e,f$) that are included in \eqref{RLTu}.
In other words, the dual of the relaxation $\widetilde{RLT_1}$ contains also   dual variables
say $\delta_{ef}$ and   $\delta_{fe}$  ($\forall e,f$), that correspond to the symmetry  constraints.
Thus, for each $e,f$ there is a constraint in the dual of the  incomplete first level RLT relaxation
that has a term $q_{ef}+(\delta_{ef}-\delta_{fe} )$.
Since $\delta_{ef}-\delta_{fe} = - (\delta_{fe}-\delta_{ef})$, this term corresponds to adding a skew symmetric matrix to the  cost matrix $Q$.
This implies that the optimal solution of the dual provides the best skew symmetric matrix that is added to the quadratic cost to improve the GL type bound.
Thus one can obtain the optimal solution for the incomplete first level RLT relaxation for the QMSTP in {\em one step} of  the GL bounding scheme.
\end{remark}

Further, in \cite{PEREIRA2015149} it is proposed a generalization of the  RLT relaxation that is based on decomposing spanning trees into forests of a fixed size.
The larger the size of the forest, the stronger the formulation.
The resulting relaxations are solved by using a Lagrangian relaxation scheme.
In \cite{Rostami2015LowerBF} it is  implemented a dual-ascent procedure to approximate the value of the (incomplete) level two RLT relaxation (${\widetilde{RLT_2}}$).

Since we couldn't find results on the first level RLT relaxation in the literature, we computed them for several instances.
In Table \ref{RLTbounds} we compare $GL$,  $\widetilde{RLT_1}$, $RLT_1$ and $\widetilde{RLT_2}$ bounds for  small CP1 instances.
For the description of instances see Section \ref{sect:nnumerics}. $UB$ stands for upper bounds.
Bounds for the (incomplete) first and second level RLT relaxation are from \cite{Rostami2015LowerBF}.
Table \ref{RLTbounds} shows that the GL type bounds are significantly weaker than the other bounds.
The results also show that difference between  $\widetilde{RLT_1}$ and $RLT_1$  bounds is not always small.
$\widetilde{RLT_2}$  bounds dominate all presented bounds in all but one case.
This discrepancy should be due to the way that those bounds are computed, i.e., by implementing  a dual ascent algorithm on the incomplete second level RLT relaxation.
For more such examples see \cite{Rostami2015LowerBF}.

\begin{table}[]
\caption{RLT type bounds}
\begin{center}
\begin{tabular}{rrr rrrr }
\hline \\
$n$ & $d$ & $UB$ & $GL$ &  $\widetilde{RLT1}$ &  $RLT_1$ &  $\widetilde{RLT_2}$   \\
\hline
10 & 33 & 350 & 299 & {\bf 350.0} & {\bf 350.0} & 344.1 \\
10 & 67 & 255 & 149 & {202.2} & 204.1 & {\bf 226.1} \\
15 & 33 & 745 & 445 & 578.2 & 603.3 & {\bf 637.8} \\
15 & 67 &  659 & 283 & 385.4 & 385.7 & {\bf 488.9} \\
20 & 33 & 1379 & 690 & 888.0 & 891.7 & {\bf 1056.7} \\
\hline
\end{tabular}\label{RLTbounds}
\end{center}
\end{table}

\section{New lower bounds for the QMSTP} \label{sect:newBounds}

\subsection{Linearization-based lower bounds} \label{sect:linBasedBnd}

\citet{Custic2015ACO} prove that the QMSTP on a complete graph  is linearizable if and only if its cost matrix  is a symmetric weak sum matrix.
In this section we exploit weak sum matrices to derive lower bounds for the QMSTP.  \citet{HuSotirov2} introduce the concept of linearization based bounds,
and show that many of the known bounding approaches including the GL type bounds are also linearization based bounds. De Meijer and Sotirov \cite{MeijerSotirov}
derive strong and efficient linearization  based bounds for the quadratic cycle cover problem.

A matrix $Q \in \mathbb{R}^{m\times m} $ is called a sum matrix if there exist $a,b  \in \mathbb{R}^{m}$ such
that $q_{ef} = a_e + b_f$ for all $e,f\in \{1,\ldots, m\}$. A weak sum matrix is a matrix for which this
property holds except for the entries on the diagonal, i.e., $q_{ef} = a_e + b_f$ for all $e \neq f$.
However for a symmetric weak sum matrix we have that vectors $a$ and $b$ are equal, i.e., $q_{ef} = a_e + a_f$  for all $e \neq f$.

It is proven in \cite[Theorem 5]{Custic2015ACO} that a  symmetric cost matrix $Q$ of the QMSTP on a complete graph is
linearizable if and only if it is a symmetric weak sum matrix. In particular, the authors show that for a symmetric weak sum matrix
of the form $q_{ef} = a_e + a_f$ and a tree  $T$ in a complete graph one has:
\[
\begin{array}{rcl}
\sum\limits_{e\in T} \sum\limits_{f\in T}  q_{ef}
 &=& \sum\limits_{e\in T} \sum\limits_{\substack{f \in T \\ f \neq e}}(a_e + a_f) + \sum\limits_{e\in T}q_{ee}\\[4ex]
 &=& 2(n-2) \sum\limits_{e\in T} a_e +  \sum\limits_{e\in T}q_{ee}\\
  &=&  \sum\limits_{e\in T} p_e,
\end{array}
\]
where $p \in \mathbb{R}^m$ with
\begin{equation}\label{pe}
p_e= 2(n-2)a_e+q_{ee}
\end{equation}
  is a linearization vector. Thus, solving the QMSTP in which  the cost matrix is  a symmetric weak sum matrix
  corresponds to solving a minimum spanning tree problem.

In \cite{HuSotirov2,MeijerSotirov} it is proven that one can obtain a lower bound for the optimal value $OPT(Q)$
of a  minimization quadratic optimization problem
with the cost matrix $Q$ from a linearization matrix  $\hat{Q}$, which satisfies $Q\geq \hat{Q}$.  In particular, we have:
\begin{align*}
OPT(Q) = \min_{x \in X}\{x^{\mathrm T} Qx\} \geq \min_{x \in X}\{x^{\mathrm T}\hat{Q}x\} = \min_{x \in X}\{x^{\mathrm T} \hat{p}\} = OPT(\hat{p}),
\end{align*}
where $\hat{p}$ is a linearization vector of $\hat{Q}$,  $X$ is the feasible set  of the optimization problem, and $OPT(\hat{p})$ is the optimal
solution of the corresponding linear problem.

Thus, for a weak sum matrix $\hat{Q}$ such that  $Q\geq \hat{Q}$ and the corresponding linearization vector $\hat{p}$,
the optimal solution for the MSTP  \eqref{eq:MSTP} with the cost vector $\hat{p}$ is a lower bound for the QMSTP  \eqref{eq:QMST}.
By maximizing the right hand side in the above inequality
over all linearization vectors of the form \eqref{pe}, one obtains the strongest linearization based bound for the QMSTP.
Therefore, the strongest linearization based lower bound for the QMST is the optimal solution of the following problem:
\begin{subequations}
\begin{align}
LBB(Q) := \max & ~~{-(n-1)\epsilon - \sum_{i=1}^n \sum_{\substack{j=1 \\ j\neq i}}^n \mu_{ij}} & \label{obj:linModel}\\
{\rm s.t.} & ~~ a_e + a_f  \leq q_{ef}  & \forall  e,f \in E, ~e\neq f \label{linConst} \\
 & ~~  \sum_{k=1}^n \theta_{ke} - \epsilon ~\leq 2(n-2)a_e+q_{ee} &\forall e \in E \label{consrEps1} \\
 & \mu_{ki} + \sum_{e \in E(\{i,j\})} \theta_{ke} ~\geq 0& k,i,j = 1,\dots,n \label{const:muSum1}\\
&  \mu_{ki} \geq 0 &\ k,i = 1,\dots,n  \label{const:mu1} \\
& \epsilon\in \mathbb{R}, ~~ \theta_{ke} \in \mathbb{R}, ~~ a_e\in \mathbb{R}   &   k= 1,\dots,n, ~e\in E.
\end{align}  \label{inBasedBoundl1}
\end{subequations}
We denote the linear programming  relaxation  \eqref{inBasedBoundl1} by $LBB$.
The $LBB$ is derived from the dual of the extended formulation of the MSTP, see \eqref{DualextendModel}.
Note that in computations one can exploit the fact that $Q=Q^{\mathrm T}$ to reduce the number of constraints in \eqref{linConst}.
 In the next section we relate relaxation \eqref{inBasedBoundl1} with a relaxation obtained from the  linear formulation of the QMSTP.

\subsection{Strong and efficient lower bounds} \label{sec:newLower}

The QMSTP formulation  \eqref{AXexact} is introduced in \cite{AssadXu}.
We exploit that QMSTP formulation to derive a sequence  of linear relaxations for the QMSTP.
First, by replacing $\mathcal{T}$ by $\mathcal{T}_{E}$, and adding the symmetry constraints $y_{ef}=y_{fe}$ ($\forall e,f$),
we obtain the following relaxation for the QMSTP:
\begin{subequations}
\begin{align}
VS_0(Q) := \min &  \sum_{e,f \in E} q_{ef} y_{ef} \\
{\rm s.t. }  &  \sum_{f \in E} y_{ef} = (n-1) x_e & \forall e \in E \label{AX2first}\\
&  y_{ee} = x_e  &  \forall e \in E  \\
&  y_{ef} =y_{fe} &  \forall e,f \in E \label{AX2sym} \\
&   (x,z) \in \mathcal{T}_{E} \label{AX2last} \\
&     0 \leq y_{ef} \leq 1 & \forall e, f \in E . \label{AX2lastY}
\end{align}\label{newmodelAX}
\end{subequations}
We denote by $VS_0$ the above relaxation.
In the above optimization problem, we  omit constraints:
\[
 \sum_{e \in E} y_{ef} = (n-1) x_f  \qquad  \forall f \in E,
\]
since they are implied by constraints \eqref{AX2first} and \eqref{AX2sym}.
Next, we compare the relaxation \eqref{newmodelAX} and the  strongest linearization based relaxation  \eqref{inBasedBoundl1}.
Let us first present the dual of the relaxation \eqref{newmodelAX}:
\begin{subequations}
\begin{align}
\max & ~~{-(n-1)\epsilon - \sum_{i=1}^n \sum_{\substack{j=1 \\ j\neq i}}^n \mu_{ij}} &  \label{objdual2}\\
 {\rm s.t.} & ~~ \alpha_e  + \gamma_e \geq  0 & \forall e \in E \label{const:gammaAlpha} \\
 & ~~ -\alpha_e  - \delta_{ef} +\delta_{fe} \leq   q_{ef} & \forall  e,f \in E, ~~e\neq f  \label{const:alphaDeltQ} \\
 & ~~  \sum_{k=1}^n \theta_{ke} + (n-1)\alpha_e+\gamma_e - \epsilon  ~\leq q_{ee} & \forall  e \in E \label{consrEps2} \\
 & ~~ \mu_{ki} + \sum_{e \in S(\{i,j\})} \theta_{ke} ~\geq 0 & k,i,j = 1,\dots,n  \label{const:muSum2} \\
& ~~ \mu_{ki} \geq 0 & k,i = 1,\dots,n  \label{const:mu2} \\
& ~~ \epsilon\in \mathbb{R}, ~~ \theta_{ke}, \alpha_e,  \gamma_e, \delta_e \in \mathbb{R}      & k= 1,\dots,n, ~e\in E.
\end{align}  \label{DualextendModel2}
\end{subequations}
To derive the above  dual problem we remove upper bounds on $x_e$ and $y_{ef}$ in the corresponding primal problem,
 since it is never optimal to set the values of these variables larger than one.
The following result shows that the strongest linearization based bound equals the bound obtained by solving \eqref{newmodelAX}, or equivalently \eqref{DualextendModel2}.

\begin{theorem} \label{prop:equivalent}
Let $G$ be a complete graph.
Then, optimization problems  \eqref{inBasedBoundl1} and \eqref{DualextendModel2} are equivalent.
\end{theorem}

\proof We show that  for every feasible solution for \eqref{DualextendModel2},
 we can find a feasible solution for \eqref{inBasedBoundl1} with the same objective value, and conversely
 for every feasible solution for \eqref{inBasedBoundl1}   we can find a feasible solution for  \eqref{DualextendModel2} with the same objective value.

Let ($\alpha$,$\gamma$,$\delta$,$\epsilon$,$\mu$,$\theta$) be a feasible solution for the optimization problem \eqref{DualextendModel2}.
Let us find a feasible solution ($\hat{a}$,$\hat{\epsilon}$,$\hat{\mu}$,$\hat{\theta}$) for \eqref{inBasedBoundl1}.
We define  $\hat{\epsilon} := \epsilon$, $\hat{\mu} := \mu$ and $\hat{\theta} := \theta$.
Thus, the objective values \eqref{objdual2} and  \eqref{obj:linModel}   are equal,
 and constraints \eqref{const:muSum1} (resp., \eqref{const:mu1})  correspond to constraints  \eqref{const:muSum2} (resp., \eqref{const:mu2}).

Let us define $\hat{a}_e := - \frac{1}{2(n-2)}((n-1)\alpha_e + \gamma_e )$  for $e \in E$
and note that
\[
\hat{a}_e = - \frac{(n-1)\alpha_e}{2(n-2)} - \frac{\gamma_e}{2(n-2)} \leq - \frac{(n-1)\alpha_e}{2(n-2)} + \frac{\alpha_e}{2(n-2)} = - \frac{\alpha_e}{2},
\]
where $\alpha_e \geq -\gamma_e$ follows from \eqref{const:gammaAlpha}.
Further from \eqref{const:alphaDeltQ} we have that
$$
\begin{array}{c}
    -\alpha_e - (\delta_{ef} - \delta_{fe}) \leq q_{ef} \\
    -\alpha_f - (\delta_{fe} - \delta_{ef}) \leq q_{fe},
\end{array}
$$
for every pair of edges  $(e,f)$ such that $e \neq f$.
By adding these two inequalities, we obtain $-(\alpha_e + \alpha_f) \leq q_{ef} + q_{fe}$.
 However, the inequality $\hat{a}_e \leq -\frac{\alpha_e}{2}$ implies
 $\hat{a}_e + \hat{a}_f \leq - \frac{1}{2}(\alpha_e + \alpha_f) \leq \frac{1}{2}(q_{ef} + q_{fe})=q_{ef}$  for every pair $(e,f)$ such that $e \neq f$.
 Note that we assumed that $Q = Q^{\mathrm T}$.
 Thus  constraints \eqref{linConst} are also satisfied.
 To show that constraints \eqref{consrEps1} are satisfied, we use \eqref{consrEps2} from where it follows:
 \[
 \sum_{k=1}^n \theta_{ke}  - \epsilon  ~\leq q_{ee}  - (n-1)\alpha_e - \gamma_e = q_{ee} + 2(n-2) \hat{a}_e \qquad \forall e\in E.
 \]

 Conversely, let ($\hat{a}$,$\hat{\epsilon}$,$\hat{\mu}$,$\hat{\theta}$) be a feasible solution for \eqref{inBasedBoundl1}.
Below, we  construct a feasible solution ($\alpha$,$\gamma$,$\delta$,$\epsilon$,$\mu$,$\theta$) for  \eqref{DualextendModel2}.
 We define $\epsilon := \hat{\epsilon}$, $\mu := \hat{\mu}$ by $\theta := \hat{\theta}$.
 Thus, we have that the objective values \eqref{objdual2} and  \eqref{obj:linModel}   are equal,
 and constraints \eqref{const:muSum1} (resp., \eqref{const:mu1})  correspond to constraints  \eqref{const:muSum2} (resp., \eqref{const:mu2}).
We define $\alpha_e := -2 \hat{a}_e$ and $\gamma_e := 2 \hat{a}_e,$ for all $e \in E$, from where it follows
$\alpha_e + \gamma_e = 0$. Thus, constraint \eqref{const:gammaAlpha} is satisfied.

From the definitions of $\alpha$ and $\gamma$ we have  $ -(n-1) \alpha_e - \gamma_e= 2(n-2) \hat{a}_e $,
and from \eqref{consrEps1} it follows
\[
 \sum_{k=1}^n \theta_{ke} - \epsilon ~\leq  q_{ee} +  2(n-2)\hat{a}_e = q_{ee} -(n-1) \alpha_e - \gamma_e \qquad e\in E.
\]
Thus,  for each $e\in E$ constraint \eqref{consrEps2} is satisfied.

It remains to verify constraints \eqref{const:alphaDeltQ}. For that purpose we consider constraints \eqref{linConst} for
a pair of edges $(e,f)$ such that $e \neq f$.
After multiplying the constraint with two, and using $q_{ef}=q_{fe}$  and $\alpha_e = -2 \hat{a}_e$ we obtain
\[
-{\alpha}_e - q_{ef} \leq q_{fe} + {\alpha}_f.
\]
Thus, there exist $\delta_{ef}$ and $\delta_{fe}$ such that
$$
\begin{array}{c}
    - \alpha_e - q_{ef}   \leq  (\delta_{ef} - \delta_{fe}) \\[1ex]
     (\delta_{ef} - \delta_{fe}) \leq q_{fe} + \alpha_f.
\end{array}
$$
This finishes the proof.  \qed

Theorem \ref{prop:equivalent} shows that the linearization based relaxation for the QMSTP is equivalent to the relaxation of an exact linear formulation for the QMSTP.
Our preliminary numerical results show that the bound $VS_0(Q)$ is not dominated by $GL(Q)$, or vice versa.

To strengthen the relaxation   $VS_0$, see \eqref{newmodelAX}, one can add the following facet defining
inequalities of the Boolean Quadric Polytope, see e.g., \cite{Padberg}:
\begin{subequations}
\begin{align}
0 \leq y_{ef} \leq x_e  \label{bqp1}\\[1ex]
x_{e} + x_{f} \leq 1 + y_{ef} \label{bqp2} \\[1ex]
y_{eg} + y_{fg} \leq x_{g} + y_{ef}  \label{bqp3} \\[1ex]
x_e + x_{f} + x_{g} \leq y_{ef} + y_{eg} + y_{fg} +1, \label{bqp4}
\end{align} \label{BQP}
\end{subequations}
where $e,f,g \in E$, $e\neq f\neq g$.
Table \ref{VS_relaxations} introduces relaxations with increasing complexity that are obtained
by adding subsets of the BQP inequalities  to  the linear program  \eqref{newmodelAX}.

\begin{table}[h!]
\begin{center}
\caption{VS relaxations}
\begin{tabular}{|c|c|c|}
\hline
name & constraints & complexity \\\hline\hline
$VS_0$ & \eqref{AX2first} -- \eqref{AX2lastY} & ${\mathcal O}(n^3+m^2)$\\\hline
$VS_1$ & \eqref{AX2first} -- \eqref{AX2last}, \eqref{bqp1} -- \eqref{bqp2} & ${\mathcal O}(n^3+ m^2 )$\\\hline
$VS_2$ & \eqref{AX2first} -- \eqref{AX2last}, \eqref{bqp1} -- \eqref{bqp4} & ${\mathcal O}(n^3+m^3)$\\\hline
\end{tabular}\label{VS_relaxations}
\end{center}
\end{table}
To solve the relaxations $VS_1$ and $VS_2$  we use a cutting plane scheme that iteratively adds  the most violated inequalities,
see section on numerical results for more details.

The following result relates  $\widetilde{RLT_1}$ and $VS_1$ relaxations.
\begin{proposition} \label{compareRLTVS1}
The  relaxation  $VS_1$ is not dominated by  the incomplete first level RLT relaxation \eqref{RLTu}, or vice versa.
\end{proposition}

\begin{proof}
We consider a feasible solution $(x,Y)$  of \eqref{RLTu} for an instance on complete graph with $6$ vertices.
Let $x:=( 0.75,  0.75, 0, 0,  0,  0,  0,  0,  0.75,  1,  0,  0.75,  1,  0,  0 )^{\mathrm T}$ and define the $15 \times 15$ symmetric matrix $Y$
whose diagonal elements correspond to  the elements of vector $x$. Further, elements on the  positions $y_{ef}=y_{fe}$ where
\[
\begin{array}{c|cccccccccccc}
e & 1&     2  &   1   &  2 &    9&     1  &   2 &   10&     1 &    2   &  9 &   12 \\    \hline
f & 9&     9  &  10   & 10 &   10&    12  &  12 &   12&    13 &   13   & 13 &   13
\end{array}
\]
equal $0.75$, while $y_{10,13}=1$, and all other elements are zero. By direct verification it follows that $(x,Y)$ is feasible for   \eqref{RLTu}.
On the other hand, we have that  $1 + y_{12} - y_{11} - y_{22} < 0$. Thus, \eqref{bqp2} are violated and $(x,Y)$ is not feasible for $(VS_1)$.
 The (incomplete) RLT bound for this particular instance is 386.5, while $VS_1(Q)=372.4$

Conversely, in Section \ref{sect:nnumerics} we  provide examples where the optimal values of $VS_1$
are strictly greater than the optimal values of $\widetilde{RLT}_1$.
\end{proof}
The  results of the previous proposition  are not surprising since constraints \eqref{bqp2} are not included in the relaxation $\widetilde{RLT_1}$.
It is not difficult to show the following results.
\begin{corollary} \label{compareRLTVS1}
The  relaxation  $VS_1$ is  dominated by  the first level RLT relaxation for the QMSTP.
\end{corollary}
\begin{corollary} \label{compareRLTVS2}
The  relaxation  $VS_2$ is not dominated by  the first level RLT relaxation for the QMSTP, or vice versa.
\end{corollary}
Note that Corollary \ref{compareRLTVS1} follows from the discussion in Section \ref{sect:RLTold}, and
 Corollary \ref{compareRLTVS2} from the fact that constraints  \eqref{bqp4} are not present in the $ RLT_1$ relaxation.

\section{Numerical results} \label{sect:nnumerics}

In this section we compare several lower bounds from the literature with the bounds introduced in Section \ref{sect:newBounds}.
Numerical experiments are performed on an Intel(R) Core(TM) i7-9700 CPU, 3.00 GHz with 32 GB memory.
 We implement our bounds in the Julia Programming Language \cite{bezanson2017julia} and use CPLEX 12.7.1.

To solve $VS_2$ relaxation  we implement  a cutting plane algorithm that starts from $VS_1$ relaxation and iteratively adds the most violated $n\cdot m$ cuts.
The algorithm stops if no more violated cuts or after two hours.

To compute upper bounds for the here introduced benchmark instances SV we implement the tabu search algorithm
and variable neighbourhood search algorithm from \cite{Cordone2012SolvingTQ}.
In the mentioned paper the authors suggest restarting the tabu search algorithm for better performance of the algorithm.
We notice that a small number of restarts (i.e., at most 5) and then  running the  neighbourhood search algorithm  can be beneficial for large instances.
The total number of iterations of the tabu search algorithm is 5000.

\subsection{Test instances}

We test our bounds on the following benchmark  sets.

The benchmark set CP is introduced by Cordone and Passeri \cite{Cordone2012SolvingTQ} and consists of 108 instances.
These instances consist of graphs with $n\in \{ 10, 15,\ldots, 50 \}$ vertices and densities $d\in \{ 33\%, 67 \%, 100 \% \}$.
There are four types of random instances denoted by CP1, CP2, CP3, CP4.
In CP1 instances, the linear and the quadratic costs are  uniformly distributed at random in  $\{1,\ldots, 10\}$.
In CP2 instances, the linear (resp., quadratic) costs are  uniformly distributed at random in  $\{1,\ldots, 10 \}$ (resp., $\{1,\ldots, 100 \}$).
In CP3 instances, the linear (resp., quadratic) costs   are  uniformly distributed at random in  $\{1,\ldots, 100 \}$ (resp., $\{1,\ldots, 10 \}$).
In CP4 instances, the linear and the quadratic costs are  uniformly distributed at random in  $\{1,\ldots, 100\}$.

The benchmark set OP is introduced by \citet{ONCAN20101762} and consists of 480 instances.
These instances consist of complete graphs with $n\in \{ 6,7,\ldots, 17,18, 30, 30, 40, 50 \}$ vertices
and are divided into three types.
In particular, in OPsym  the linear (resp., quadratic)  costs are uniformly distributed at random in $\{1,\ldots, 100\}$ (resp., $\{1,\ldots, 20\}$).
In OPvsym the linear costs are uniformly distributed at random in $\{1,\ldots, 10 000\}$, while
the quadratic costs  are obtained associating to the vertices $i\in V$ random values $w(i)$ uniformly distributed in $\{1,\ldots, 10\}$ and
setting $q_{ \{ i,j \},\{k,l \} } = w(i)w(j)w(k)w(l)$.
In OPesym the vertices are spread uniformly at random in a square of length side 100,
the linear costs are the Euclidean distances between the end vertices of each edge,
while  the quadratic costs are the Euclidean distances between the midpoints of the edges.

We introduce the benchmark set SV\footnote{Interested reader can download SV instances from the following link:\\
 \url{https://drive.google.com/drive/folders/1bpF8AfAn2K5QGSoB6Yl76znw6PwdvItb} }. Our benchmark consists of 24 instances.
For given a size $n \in \{10,12,14,16,18,20,25,30\}$, density $d \in \{ 33\%, 67 \%, 100 \% \}$,
a maximum cost for the diagonal entries, and a maximum cost for the off-diagonal entries, we generate an instance in the following way.
$10 \%$ of the rows are randomly chosen.  These rows will have high costs with each other (between 90\% and 100\% of the maximum off-diagonal cost),
and low costs with rest (between 20\% and 40\% of the maximum off-diagonal cost).
The rows that are not selected have an interaction cost of between 50\% and 70\% of the maximum off-diagonal cost.
Finally, the impact of diagonal entries is greatly minimized, between 0 and 20\% of the maximum diagonal cost.
The cost matrices obtained this way are not symmetric, but they can be made so at the user's convenience.

\subsection{Computational Results}

We first present  results for our  benchmark set {SV}.
Table \ref{VS} reads as follows. In the first two columns we list the number of vertices and density of a graph, respectively.
In the third column we provide upper bounds computed as mentioned earlier.
In the following three columns we list the incomplete first level RLT  bound  $\widetilde{RLT_1}$ see \eqref{RLTu},
$VS_1$ bound that is \eqref{AX2first} -- \eqref{AX2last}, \eqref{bqp1} -- \eqref{bqp2},
and $VS_2$ bound that is \eqref{AX2first} -- \eqref{AX2last}, \eqref{bqp1} -- \eqref{bqp4}, see also Table \ref{VS_relaxations}.
In columns 7 -- 9 we present gaps by using the formula $100 \cdot(UB-LB)/UB$ where $LB$ stands for the value of the lower bound.
Note that the gap we present in our numerical results differs from the gap used in other QMSTP papers, where the authors use  $100\cdot (UB-LB)/LB$.
In the last two columns of Table \ref{VS} we presents time in seconds required to solve our relaxations.
We do not report time for computing  $\widetilde{RLT_1}$ as we implement \eqref{RLTu} directly,
 while the other authors that compute RTL type bounds use a more efficient way to compute those bounds.

The results in Table \ref{VS} show that  $VS_1$ bounds  are stronger than  $\widetilde{RLT_1}$  bounds for all instances
 except for the SV instance with $n=25$ and $d=100$, and  the SV instance with $n=30$ and $d=100$.
The difference between gaps for $\widetilde{RLT_1}$  and $VS_1$ bounds for  both instances is less than $1\%$.
Moreover $VS_2$ relaxation  provides a better bound than $\widetilde{RLT_1}$  for the {SV} instance with $n=25$ and $d=100$
within 2 minutes of the  cutting plane algorithm.
Similarly, $VS_2$ relaxation  provides a better bound than $\widetilde{RLT_1}$  relaxation  for
the {SV} instance $n=30$ and $d=100$ after a few iterations of the cutting plane algorithm.
Recall that the $VS_1$ relaxation is not dominated by  the incomplete first level RLT relaxation, or vice versa   in general, see Proposition \ref{compareRLTVS1}.
Therefore our benchmark set {SV} can be used to test the quality of  QMSTP bounds that are not  RTL type bounds.
Table \ref{VS} shows that computational times for solving the relaxation $VS_1$ are very small for all instances.
The computational effort for solving the relaxation $VS_2$ is  small for instances with $n\leq 25$ and $d \in \{ 33\%, 67 \% \}$.
 The results also show that it is computationally  more challenging to solve $VS_2$ for dense instances.
Note also that we can stop our cutting plane algorithm at any time and obtain a valid lower bound.

{\tiny
\begin{table}[]
\caption{SV instances: bounds and gaps}
\label{VS}
\begin{center}
\begin{tabular}{rrr rrr c rrr c rr}
\hline
\multicolumn{2}{c}{instance} & &  \multicolumn{3}{c}{lower bounds} && \multicolumn{3}{c}{Gap ($\%$)}   && \multicolumn{2}{c}{time (s)}\\
\cline{1-2}   \cline{4-6} \cline{8-10}  \cline{12-13} \\
$n$ & $d (\%)$  & $UB$  &  $\widetilde{RLT_1}$ & $VS_1$  & $VS_2$  &&  $ \widetilde{RLT_1}$ & $VS_1$ & $VS_2$ && $VS_1$ & $VS_2$ \\ \hline
10  & 33   &  4217 & {\bf  4217}    &  {\bf 4217}   & {\bf 4217}    &&   {\bf  0}  & {\bf 0}   & {\bf 0} && $< 0.05$ & $< 0.05$ \\
10  & 67   & 3981  &  3752.6  &  3876.8 & {\bf 3981}    &&  5.8  & 2.6 & {\bf 0} && $< 0.05$ & 0.3\\
10  & 100  & 3930  &  3499.1  &  3604.3 & {\bf 3857}    &&  11.0   & 8.2 & {\bf 1.9} && 0.1 & {1.9}  \\
12  & 33   & 6141  &  5859.4  &  6058.1 & {\bf 6136}    &&  4.6  & 1.3 & {\bf 0.1} && 0.0 & {0.2}\\
12  & 67   & 6050  &  5759.6  &  5910.4 & {\bf 6038}    &&  4.8  & 2.3 & {\bf 0.2} && {0.1} & 3.1 \\
12  & 100  & 6051  &  5726    & 5868.1  & {\bf 5991.8}  &&  5.4  & 3.0 & {\bf 1.0} && {0.3} & 5.0\\
14  & 33   & 8736  &  8495.2  &  8619.1 & {\bf 8709.1}  &&  2.8  & 1.3 & {\bf 0.3} && 0.1 & 0.4 \\
14  & 67   & 8606  &  8096.8  &  8306.0   & {\bf 8542.8}  &&  5.9  & 3.5 & {\bf 0.7}  && 0.2 &4.8 \\
14  & 100  & 8513  &  7556    &  7731.7 & {\bf 8267.4}  && 11.2  & 9.2 & {\bf 2.9} && 2.4 & 61.1\\
16  & 33   & 11735 &  11231.5 &  11465  & {\bf 11711.5} &&  4.3  & 2.3 & {\bf 0.2} && 0.1 & 1.0\\
16  & 67   & 11610 &  10559.8 &  10816.2 & {\bf 11335.3} && 9.0  & 6.8 & {\bf 2.4} && 1.7 & 50.4 \\
16  & 100  & 11516 &  10089.6 &  10302.7 & {\bf 11003.5} && 12.4 & 10.5 & {\bf 4.5} && 8.4 & 241.1 \\
18  & 33   & 15125 &  13995.8 &  14382.9 & {\bf 14916.5} && 7.5  & 4.9  & {\bf 1.4} && 0.3 & 3.6 \\
18  & 67   & {15019} & 12976.9  &  13210.6 & {\bf 14204.6} && 13.6 & 12.0 & {\bf 5.4} && 3.2 & 200.1 \\
18  & 100  & 14943 & 13217.6  & 13501.8  & {\bf 14336.0}   && 11.5 & 9.6 & {\bf 4.1} && 2.6 & 89.6\\
20  & 33   & 19057 & 17777.2  & 18178.9  & {\bf 18778.2}  && 6.7 & 4.6 & {\bf 1.5} && 0.3 & 7.1 \\
20  & 67   & 18830 & 16325.5  & 16433.5  & {\bf 17130.4}  && 13.3& 12.7 & {\bf 9.2} && 4.8 & 163.2 \\
20  & 100  & 18812 & 16026    & 16299.6  & {\bf 17528.9}  && 14.8 &13.4 & {\bf 6.8} && 4.1& 185.15 \\
25  & 33   & 30747 & 28436.4  & 29102.6  & {\bf 30084.4} && 7.5  &  5.3   &  {\bf 2.2} &&  4.8   &  179.1 \\
25  & 67   & 30554 & 27061.5  & 27546.6  & {\bf 29186.1} &&   11.4    & 9.8    &  {\bf 4.5}    &&  2.8   & 377.9  \\
25  & 100  & 30405 & 24455.6  & 24257.7  & {\bf 26251.7}  &&  19.7    &  20.3   & {\bf 13.7}  &&   10.4   & 1260.7 \\
30  & 33   & 45184 & 40946.1  &  41995.1 & {\bf 43889.3} && 7.1 &  9.4   & {\bf 2.7}   && {34.3}  &  2268.6 \\
30  & 67   & 44989  & 37676.7  & 38089.8  & {\bf 41162.8} &&  16.25   &  15.3   &  {\bf 8.5}   &&  13.5    & 1359.9 \\
30  & 100  & {44847} & 35116.9  &  35037.5  &  {\bf 37489.8}   &&  21.7     & 21.9    &  {\bf 16.4}    && 41.8    &  2700.0 \\
\end{tabular}
\end{center}
\end{table}
}

Tables \ref{CP1} -- \ref{Sym} read similarly  to Table \ref{VS}.
We do not report results for OPvsym instances since gaps for $VS_1$ bounds for those instances  are less than or equal to $0.2 \%$.
Also, we do not present results for instances with $n>35$ since the corresponding gaps are (too) large for all bounds in the literature, including ours.
Clearly, it is a big challenge to obtain good bounds for the QMSTP when $n \geq 20$ for most of the instances and approaches.

Running times required to solve $VS_1$ relaxation for  most of the  test instances in Tables \ref{CP1} -- \ref{Sym} are similar to  times given in Table \ref{VS}.
That is, for instances CP1, CP2, CP3 and CP4 with  $n\leq 20$ and $d \in \{ 33\%, 67 \%, 100\% \}$  as well as instances with $n=25$ and $d \in \{ 33\%, 67 \% \}$
computation times  required to solve $VS_1$ relaxation are a few seconds.
Running times for solving CP1, CP2, CP3 and CP4 instances with $n\in \{30,35\}$ and $d \in \{ 33\%, 67 \% \}$  do not exceed  76 seconds per instance.
For the remaining CP instances,  largest computational time required for solving $VS_1$ relaxation is for CP3 instance  with  $n=35$ and $d=100$, and that is 4 minutes.
 $VS_1$  bounds for instances OPsym and OPesym with $n\leq 20$ are  computed in a few seconds, and for $n=30$ in less than 134 seconds.

Computational time used to compute $VS_2$ bound for each CP instance with $n=30$, $d=100$ and $n=35$, $d \in \{ 67\%, 100\% \}$  is  two hours.
Note that we allow our algorithm to run at most two hours.
Computational times required for computing $VS_2$ bounds for other instances in Tables~\ref{CP1} -- \ref{Sym}  are comparable to running  times in Table~\ref{VS}.

Results in Table \ref{CP1} -- \ref{CP4}  present bounds for CP instances and show that
the GL bounds are significantly weaker than the other listed bounds. The results also show that
$\widetilde{RLT_1}$ bounds  are stronger than $VS_1$ bounds for all instances, while
$VS_2$ bounds dominate $\widetilde{RLT_1}$ bounds for all instances except the CP1 instance with $n=35$ and $d=100$.
The results also show that $VS_2$ bounds are not dominated by $\widetilde{RLT_2}$ bounds  or vice versa.

Numerical results in Table \ref{Esym} provide bounds  for  OPesym instances and
show that $VS_1$ and $VS_2$ bounds are weaker  than  $\widetilde{RLT_1}$ and  $\widetilde{RLT_2}$ bounds  for all OPesym instances.
This is due to the fact that  there are not many violated cuts of type \eqref{bqp1} -- \eqref{bqp4}  for those instances.

Table \ref{Sym} presents results for  OPsym instances  and shows that $VS_2$ bounds are stronger  than  $\widetilde{RLT_2}$ bounds for instances with less than 16 vertices.
Moreover,  $VS_2$ bounds are stronger  than   $\widetilde{RLT_1}$ bounds for all presented instances except for $n=18$.
 Note that results in Table \ref{Esym} and Table \ref{Sym} present an average over 10 instances of a given size.

\citet{PEREIRA2015149} derive strong bounds for the QMSTP by using the idea of partitioning spanning trees into forests of a given fixed size.
The resulting model  can be seen as a generalization of the RLT relaxation.
To obtain bounds, the authors introduce a bounding procedure based on  Lagrangian relaxation.
 In \cite{PEREIRA2015149}, Table 1 the authors provide  bounds for  CP instances and  $n=25$. Our lower bound $VS_2$ is better than their bound only for CP3 instance with $n=25$ density $33\%$.
Namely,  \citet{PEREIRA2015149}  report bound 2652.5  and we  compute 2730.3.
Although bounds that result from the generalized RLT relaxation are stronger than the bounds obtained from the incomplete first level RLT relaxation,
the former  were not used within a Branch and Bound algorithm in \cite{PEREIRA2015149}. Instead, the authors use $\widetilde{RLT_1}$ relaxation
to solve several instances to optimality.
Note that our $VS_2$ bounds are stronger than $\widetilde{RLT_1}$ bounds for all CP instances except for the CP1 instance with $n = 35$ and $d = 100$,
and all   OPsym  instances except for $n=18$.

{\tiny
\begin{table}[]
\caption{CP1 instances: bounds and gaps}
\label{CP1}
\begin{tabular}{rrrrrrrrr c rrrr}
\hline
\multicolumn{2}{c}{instance} & &  \multicolumn{5}{c}{lower bounds} & \multicolumn{5}{c}{Gap ($\%$)}  \\
\cline{1-2}   \cline{4-8} \cline{10-14} \\
$n$ & $d (\%)$  & $UB$ & $GL$ & $\widetilde{RLT_1}$ & $\widetilde{RLT_2}$ & $VS_1$ & $VS_2$ & & $GL$ & $ \widetilde{RLT_1}$ & $ \widetilde{RLT_2}$  & $VS_1$ & $VS_2$ \\
\hline
10 & 33 & 350 & 299 & {\bf 350} & 344,1 & {\bf 350} & {\bf 350} &  & 14.6 & {\bf 0} & 1.7 &{\bf 0} & {\bf 0} \\
10 & 67 & 255 & 149 & 202.2 & 226.1 & 166.4 & {\bf 248.8} &  & 41.6 & 20.7 & 11.3 & 34.7 & {\bf 2.4} \\
10 & 100 & 239 & 120 & 159.7 & 199.0 & 140.7 &{\bf 201.3} &  & 49.8 & 33.2 & 16.7 & 41.1 & {\bf 15.8} \\
 &  &  &  &  &  &  &  &  &  &  &  &  &  \\
15 & 33 & 745 & 445 & 578.2 & 637.8 & 487 & {\bf 709.2} &  & 40.3 & 22.4 & 14.4 & 34.6 & {\bf 4.8} \\
15 & 67 & 659 & 283 & 385.4 & {\bf 488.9} & 324.4 &  478.9 &  & 57.1 & 41.5 & {\bf 25.8} & 50.8 & 27.3 \\
15 & 100 & 620 & 246 & 320.9 & {\bf 442.9} & 281.5 & 386.1 &  & 60.3 & 48.2 & {\bf 28.6} & 54.6 & 37.7 \\
 &  &  &  &  &  &  &  &  &  &  &  &  &  \\
20 & 33 & 1379 & 690 & 888 & 1056.7 & 740.9 & {\bf 1165.4} &  & 50.0 & 35.6 & 23.4 & 46.3 & {\bf 15.5} \\
20 & 67 & 1252 & 454 & 603.1 & {\bf 843.1} & 523.3 & 742.7 &  & 63.7 & 51.8 & {\bf 32.7} & 58.2 & 40.7 \\
20 & 100 & 1174 & 398 & 506.9 & {\bf 737.9} & 436.7 & 587.7 &  & 66.1 & 56.8 & {\bf 37.1} & 62.8 & 49.9 \\
 &  &  &  &  &  &  &  &  &  &  &  &  &  \\
25 & 33 & 2185 & 985 & 1285.1 & 1594.9 & 1045.3 & {\bf  1615.2} &  & 54.9 & 41.2 & 27.0 & 52.2 &  {\bf 26.1} \\
25 & 67 & 2023 & 660 & 834.9 & {\bf 1239.6} & 717 & 1003.7 &  & 67.4 & 58.7 & {\bf 38.7} & 64.6 & 50.4 \\
25 & 100 & 1943 & 596 & 719.1 & {\bf 1091.0} & 645.6 & 823.4 &  & 69.3 & 63.0 & {\bf 43.9} & 66.8 & 57.6 \\
 &  &  &  &  &  &  &  &  &  &  &  &  &  \\
30 & 33 & 3205 & 1260 & 1617.9 & {\bf 2149.6} & 1352.6 & 2062.12 &  & 60.7 & 49.5 & {\bf  32.9} & 57.8 & 35.7 \\
30 & 67 & 2998 & 916 & 1118.2 & {\bf 1681.4} & 956.6 & 1309.6 &  & 69.4 & 62.7 & {\bf 43.9} & 68.1 & 56.3 \\
30 & 100 & 2874 & 854 & 986.3 & {\bf 1495.3} & 906.9 & 1069.5 &  & 70.3 & 65.7 & {\bf 48.0} & 68.4 & 62.8 \\
 &  &  &  &  &  &  &  &  &  &  &  &  &  \\
35 & 33 & 4474 & 1597 & 2014.4 & {\bf 2812.1} & 1690.5 & 2529.2 &  & 64.3 & 55.0 & {\bf 37.1} & 62.2 & 43.5 \\
35 & 67 & 4147 & 1215 & 1437.3 & {\bf 2208.2} & 1262.7 & 1653.2 &  & 70.7 & 65.3 & {\bf 46.8} & 69.6 & 60.1 \\
35 & 100 & 4000 & 1156 & 1303.8 & {\bf 1953.1} & 1222.1 & 1286.4 &  & 71.1 & 67.4 &{\bf  51.2} & 69.4 & 67.8
\end{tabular}
\end{table}
}

{\tiny
\begin{table}[]
\caption{CP2 instances: bounds and gaps}
\label{CP2}
\begin{tabular}{rrrrrrrrr c rrrr}
\hline
\multicolumn{2}{c}{instance} & &  \multicolumn{5}{c}{lower bounds} & \multicolumn{5}{c}{Gap ($\%$)}  \\
\cline{1-2}   \cline{4-8} \cline{10-14} \\
$n$ & $d (\%)$  & $UB$ & $GL$ & $\widetilde{RLT_1}$ & $ \widetilde{RLT_2}$ & $VS_1$ & $VS_2$ & & $GL$ & $ \widetilde{RLT_1}$ & $ \widetilde{RLT_2}$  & $VS_1$ & $VS_2$ \\
\hline
10 & 33 & 3122 & 2562 & {\bf 3122} & 3045.9 & 3114.6 & {\bf 3122} &  & 17.9 & {\bf 0} & 2.4 & 0.2 & {\bf 0} \\
10 & 67 & 2042 & 809 & 1399.6 & 1710.2 & 1014.9 & {\bf 1946.2} &  & 60.4 & 31.5 & 16.2 & 50.3 & {\bf 4.7} \\
10 & 100 & 1815 & 553 & 937 & 1384.4 & 733 & {\bf 1388.4} &  & 69.5 & 48.4 & 23.7 & 59.6 & {\bf 23.5} \\
 &  &  &  &  &  &  &  &  &  &  &  &  &  \\
15 & 33 & 6539 & 3272 & 4684.1 & 5329.3 & 3690.9 & {\bf 6133.7} &  & 50.0 & 28.4 & 18.5 & 43.6 & {\bf 6.2} \\
15 & 67 & 5573 & 1555 & 2589.7 & {\bf 3760.5} & 1908.1 & 3637 &  & 72.1 & 53.5 & {\bf 32.5} & 65.8 & 34.7 \\
15 & 100 & 5184 & 1070 & 1829.9 &{\bf  3236.0} & 1406 & 2581.3 &  & 79.4 & 64.7 & {\bf 37.6} & 72.9 & 50.2 \\
 &  &  &  &  &  &  &  &  &  &  &  &  &  \\
20 & 33 & 12425 & 4801 & 7035.7 & 8849.7 & 5390.8 & {\bf 10106.7} &  & 61.4 & 43.4 & 28.8 & 56.6 & {\bf 18.7} \\
20 & 67 & 10893 & 2352 & 3816.7 & {\bf 6573.9} & 2907.5 & 5415.4 &  & 78.4 & 65.0 & {\bf 39.6} & 73.3 & 50.3 \\
20 & 100 & 10215 & 1676 & 2748.2 & {\bf 5442.2} & 2060.2 & 3720.7 &  & 83.6 & 73.1 & {\bf 46.7} & 79.8 & 63.6 \\
 &  &  &  &  &  &  &  &  &  &  &  &  &  \\
25 & 33 & 19976 & 6642 & 10068.5 & 13460.9 & 7361.8 & {\bf 13722.1} &  & 66.8 & 49.6 & 32.6 & 63.1 & {\bf 31.3} \\
25 & 67 & 18251 & 3196 & 5054.3 & {\bf 9666.8} & 3795.7 & 7015.5 &  & 82.5 & 72.3 & {\bf 47.0} & 79.2 & 61.6 \\
25 & 100 & 17411 & 2123 & 3469 & {\bf 7921.3} & 2623.2 & 4759.9 &  & 87.8 & 80.1 & {\bf 54.5} & 84.9 & 72.7 \\
 &  &  &  &  &  &  &  &  &  &  &  &  &  \\
30 & 33 & 29731 & 7953 & 12046.6 & {\bf 17942.7} & 9120.5 & 17066.6 &  & 73.3 & 59.5 & {\bf 39.6} & 69.3 & 42.6 \\
30 & 67 & 27581 & 3991 & 6382 & {\bf 12840.3} & 4684.5 & 8621.6 &  & 85.5 & 76.9 &{\bf  53.4} & 83.0 & 68.7 \\
30 & 100 & 26146 & 2712 & 4332.4 & {\bf 10624.1} & 3271.8 & 5880 &  & 89.6 & 83.4 &{\bf  59.4} & 87.5 & 77.5 \\
 &  &  &  &  &  &  &  &  &  &  &  &  &  \\
35 & 33 & 42305 & 9995 & 14684.1 & {\bf 23568.2} & 10771 & 20378.3 &  & 76.4 & 65.3 & {\bf 44.3} & 74.5 & 51.8 \\
35 & 67 & 38490 & 4778 & 7409 & {\bf 16274.8} & 5517.8 & 10165.3 &  & 87.6 & 80.8 & {\bf 57.7} & 85.7 & 73.6 \\
35 & 100 & 36723 & 3388 & 5241.7 & {\bf 13687.3} & 4016.8 & 6782.8 &  & 90.8 & 85.7 & {\bf 62.7} & 89.1 & 81.5
\end{tabular}
\end{table} }

{\tiny
\begin{table}[]
\caption{CP3 instances: bounds and gaps}
\label{CP3}
\begin{tabular}{rrrrrrrrr c rrrr}
\hline
\multicolumn{2}{c}{instance} & &  \multicolumn{5}{c}{lower bounds} & \multicolumn{5}{c}{Gap ($\%$)}  \\
\cline{1-2}   \cline{4-8} \cline{10-14} \\
$n$ & $d (\%)$  & $UB$ & $GL$ & $\widetilde{RLT_1}$ & $ \widetilde{RLT_2}$ & $VS_1$ & $VS_2$ & & $GL$ & $ \widetilde{RLT_1}$ & $ \widetilde{RLT_2}$  & $VS_1$ & $VS_2$ \\
\hline
10 & 33 & 646 & 589 & {\bf 646} & {\bf 646} & {\bf 646} & {\bf 646} &  & 8.8 & {\bf 0} & {\bf 0} & {\bf 0} & {\bf 0} \\
10 & 67 & 488 & 320 & {\bf 488} & 476.1 & 474.9 & {\bf 488} &  & 34.4 & {\bf 0} & 2.4 & 2.7 & {\bf 0} \\
10 & 100 & 426 & 234 & 386.2 & 401.1 & 360.4 & {\bf 426} &  & 45.1 & 9.3 & 5.8 & 15.4 & {\bf 0} \\
 &  &  &  &  &  &  &  &  &  &  &  &  &  \\
15 & 33 & 1236 & 845 & 1180.5 & 1176.0 & 1113.5 & {\bf 1218} &  & 31.6 & 4.5 & 4.9 & 9.9 & {\bf 1.5} \\
15 & 67 & 966 & 508 & 848.1 & 886.2 & 771.7 & {\bf 965.3} &  & 47.4 & 12.2 & 8.3 & 20.1 & {\bf 0.1} \\
15 & 100 & 975 & 421 & 780 & 849.3 & 724.9 & {\bf 910} &  & 56.8 & 20.0 & 12.9 & 25.7 & {\bf 6.7} \\
 &  &  &  &  &  &  &  &  &  &  &  &  &  \\
20 & 33 & 1972 & 1131 & 1672.6 & 1759.1 & 1511.9 & {\bf 1911.6} &  & 42.6 & 15.2 & 10.8 & 23.3 & {\bf 3.1} \\
20 & 67 & 1792 & 743 & 1307.1 & 1470.1 & 1190 &{\bf  1541} &  & 58.5 & 27.1 & 18.0 & 33.6 & {\bf 14.0} \\
20 & 100 & 1544 & 532 & 1056.1 & 1220.6 & 949.7 & {\bf 1257.5} &  & 65.5 & 31.6 & 20.9 & 38.5 & {\bf 18.6} \\
 &  &  &  &  &  &  &  &  &  &  &  &  &  \\
25 & 33 & 2976 & 1448 & 2289.2 & 2488.3 & 2071.7 & {\bf 2730.3} &  & 51.3 & 23.1 & 16.4 & 30.4 & {\bf 8.3} \\
25 & 67 & 2546 & 888 & 1630 & 1917.2 & 1468.8 & {\bf 1932.1} &  & 65.1 & 36.0 & 24.7 & 42.3 & {\bf 24.1} \\
25 & 100 & 2471 & 761 & 1409.6 & {\bf 1740.1} & 1284.1 & 1657.5 &  & 69.2 & 43.0 & {\bf 29.6} & 48.0 & 32.9 \\
 &  &  &  &  &  &  &  &  &  &  &  &  &  \\
30 & 33 & 4070 & 1834 & 2856.1 & 3235.3 & 2574.2 & {\bf 3383.1} &  & 54.9 & 29.8 & 20.5 & 36.8 & {\bf 16.9} \\
30 & 67 & 3649 & 1152 & 2053.5 & {\bf 2516.6} & 1854.9 & 2425.4 &  & 68.4 & 43.7 & {\bf 31.0} & 49.2 & 33.5 \\
30 & 100 & 3483 & 986 & 1776.1 & {\bf 2257.3} & 1627.5 & 2048.2 &  & 71.7 & 49.0 & {\bf 35.2} & 53.3 & 41.2 \\
 &  &  &  &  &  &  &  &  &  &  &  &  &  \\
35 & 33 & 5423 & 2060 & 3360 & 3946.9 & 3007.1 & {\bf 4016} &  & 62.0 & 38.0 & 27.2 & 44.5 & {\bf 25.9} \\
35 & 67 & 4981 & 1430 & 2515.7 & {\bf 3195.0} & 2324.1 & 2938.9 &  & 71.3 & 49.5 & {\bf 35.9} & 53.3 & 41.0 \\
35 & 100 & 4770 & 1288 & 2190.1 & {\bf  2866.6} & 1995.7 & 2504.1 &  & 73.0 & 54.1 & {\bf 39.9} & 58.2 & 47.5
\end{tabular}
\end{table} }

{\tiny
\begin{table}[]
\caption{CP4 instances: bounds and gaps}
\label{CP4}
\begin{tabular}{rrrrrrrrr c rrrr}
\hline
\multicolumn{2}{c}{instance} & &  \multicolumn{5}{c}{lower bounds} & \multicolumn{5}{c}{Gap ($\%$)}  \\
\cline{1-2}   \cline{4-8} \cline{10-14} \\
$n$ & $d (\%)$  & $UB$ & $GL$ & $\widetilde{RLT_1}$ & $ \widetilde{RLT_2}$ & $VS_1$ & $VS_2$ & & $GL$ & $ \widetilde{RLT_1}$ & $ \widetilde{RLT_2}$  & $VS_1$ & $VS_2$ \\
\hline
10 & 33 & 3486 & 2891 & 3486 & 3424.4 & {\bf 3486} & {\bf 3486} &  & 17.1 & {\bf 0} & 1.8 & {\bf 0} & {\bf 0} \\
10 & 67 & 2404 & 1158 & 1794 & 2076.0 & 1408.8 & {\bf 2330.3} &  & 51.8 & 25.4 & 13.6 & 41.4 & {\bf 3.1} \\
10 & 100 & 2197 & 823 & 1321.1 & 1743.7 & 1120.2 & {\bf 1794.7} &  & 62.5 & 39.9 & 20.6 & 49.0 & {\bf 18.3} \\
 &  &  &  &  &  &  &  &  &  &  &  &  &  \\
15 & 33 & 7245 & 3859 & 5354.8 & 6017.4 & 4371.2 & {\bf 6819.9} &  & 46.7 & 26.1 & 16.9 & 39.7 & {\bf 5.9} \\
15 & 67 & 6188 & 2003 & 3214.5 & {\bf 4339.4} & 2542.9 & 4276.8 &  & 67.6 & 48.1 & {\bf 29.9} & 58.9 & 30.9 \\
15 & 100 & 5879 & 1567 & 2490 & {\bf 3865.2} & 2056.7 & 3236.1 &  & 73.3 & 57.6 & {\bf 34.3} & 65.0 & 45.0 \\
 &  &  &  &  &  &  &  &  &  &  &  &  &  \\
20 & 33 & 13288 & 5646 & 7914.2 & 9741.9 & 6256.3 & {\bf 10966.6} &  & 57.5 & 40.4 & 26.7 & 52.9 & {\bf 17.5} \\
20 & 67 & 11893 & 2949 & 4693.4 & {\bf 7442.4} & 3791.7 & 6303.3 &  & 75.2 & 60.5 & {\bf 37.4} & 68.1 & 47.0 \\
20 & 100 & 11101 & 2103 & 3574.1 & {\bf 6219.0} & 2887 & 4556.5 &  & 81.1 & 67.8 & {\bf 44.0} & 74.0 & 59.0 \\
 &  &  &  &  &  &  &  &  &  &  &  &  &  \\
25 & 33 & 21176 & 7631 & 11186.5 & 14584.0 & 8507.1 & {\bf 14859.1} &  & 64.0 & 47.2 & 31.1 & 59.8 & {\bf 29.8} \\
25 & 67 & 19207 & 3821 & 6095.5 & {\bf 10652.8} & 4828.6 & 8061.4 &  & 80.1 & 68.3 & {\bf 44.5} & 74.9 & 58.0 \\
25 & 100 & 18370 & 2534 & 4490.3 & {\bf 8874.4} & 3649.3 & 5808.4 &  & 86.2 & 75.6 & {\bf 51.7} & 80.1 & 68.4 \\
 &  &  &  &  &  &  &  &  &  &  &  &  &  \\
30 & 33 & 31077 & 9255 & 13401 & {\bf 19278.5} & 10451 & 18403.3 &  & 70.2 & 56.9 & {\bf 38.0} & 66.4 & 40.8 \\
30 & 67 & 28777 & 4830 & 7566.9 &{\bf  14017.0} & 5935.5 & 9885.3 &  & 83.2 & 73.7 & {\bf 51.3} & 79.4 & 65.6 \\
30 & 100 & 27314 & 3198 & 5555 & {\bf 11803.8} & 4495.3 & 7120.8 &  & 88.3 & 79.7 & {\bf 56.8} & 83.5 & 73.9 \\
 &  &  &  &  &  &  &  &  &  &  &  &  &  \\
35 & 33 & 43629 & 11107 & 16170.9 & {\bf 24988.0} & 12271.2 & 21875 &  & 74.5 & 62.9 & {\bf 42.7} & 71.9 & 49.9 \\
35 & 67 & 39660 & 5631 & 8884.4 & {\bf 17970.1} & 6998.4 & 11635.8 &  & 85.8 & 77.6 & {\bf 54.7} & 82.4 & 70.7 \\
35 & 100 & 38049 & 3917 & 6707 & {\bf 15039.1} & 5477.4 & 8271 &  & 89.7 & 82.4 & {\bf 60.5} & 85.6 & 78.3
\end{tabular}
\end{table} }

{\tiny
\begin{table}[]
\caption{OPesym instances: bounds and gaps}
\label{Esym}
\begin{tabular}{rrrrrrrrr c rrrr}
\hline
\multicolumn{2}{c}{instance} & &  \multicolumn{5}{c}{lower bounds} & \multicolumn{5}{c}{Gap ($\%$)}  \\
\cline{1-2}   \cline{4-8} \cline{10-14} \\
$n$ & $d (\%)$  & $UB$ & $GL$ & $\widetilde{RLT_1}$ & $ \widetilde{RLT_2}$ & $VS_1$ & $VS_2$ & & $GL$ & $ \widetilde{RLT_1}$ & $ \widetilde{RLT_2}$  & $VS_1$ & $VS_2$ \\
\hline
6 & 100 & 541.2 & 471.8 & {\bf 539} & 538.5 & 468.2 & 475.1 &  & 12.8 & {\bf 0.4} & 0.5 & 13.5 & 12.2 \\
7 & 100 & 783.7 & 675.8 & {\bf 781.4} & 774.4 & 651.1 & 664.3 &  & 13.8 & {\bf 0.3} & 1.2 & 16.9 & 15.2 \\
8 & 100 & 1020.1 & 887 & {\bf 1016} & 1001.1 & 807.5 & 826.1 &  & 13.0 & {\bf 0.4} & 1.9 & 20.8 & 19.0 \\
9 & 100 & 1356 & 1168.1 & {\bf 1347.9} & 1326.8 & 1172.1 & 1199.9 &  & 13.9 & {\bf 0.6} & 2.2 & 13.6 & 11.5 \\
10 & 100 & 1427.1 & 1224.4 & {\bf 1420} & 1395.0 & 1156.1 & 1197.4 &  & 14.2 & {\bf 0.5} & 2.2 & 19.0 & 16.1 \\
11 & 100 & 1545.1 & 1324.6 & {\bf 1540.5} & 1508.9 & 1265.2 & 1315.9 &  & 14.3 & {\bf 0.3} & 2.3 & 18.1 & 14.8 \\
12 & 100 & 1901.6 & 1623.5 & {\bf 1894} & 1849.8 & 1568.7 & 1626.1 &  & 14.6 & {\bf 0.4} & 2.7 & 17.5 & 14.5 \\
13 & 100 & 2175.3 & 1861.1 & {\bf 2168.8} & 2114.0 & 1803.1 & 1889.5 &  & 14.4 & {\bf 0.3} & 2.8 & 17.1 & 13.1 \\
14 & 100 & 2527.9 & 2164 & {\bf 2522.9} & 2456.7 & 2028.1 & 2121.3 &  & 14.4 & {\bf 0.2} & 2.8 & 19.8 & 16.1 \\
15 & 100 & 2588.8 & 2237.9 & {\bf 2578.5} & 2515.8 & 2256.2 & 2351.1 &  & 13.6 & {\bf 0.4} & 2.8 & 12.8 & 9.2 \\
16 & 100 & 2980.1 & 2524.6 & {\bf 2962.3} & 2884.9 & 2461.6 & 2575.2 &  & 15.3 & {\bf 0.6} & 3.2 & 17.4 & 13.6 \\
17 & 100 & 3372.2 & 2837.4 & {\bf 3342.1} & 3251.9 & 2550.7 & 2667 &  & 15.9 & {\bf 0.9} & 3.6 & 24.4 & 20.9 \\
18 & 100 & 3569 & 2999.2 &{\bf  3551.2} & 3435.0 & 2844.9 & 2995.1 &  & 16.0 & {\bf 0.5} & 3.8 & 20.3 & 16.1 \\
30 & 100 & 8056.7 & 6608.5 & {\bf 7922.0} & 7579.2 & 6299.5 & 6753.0 &  & 18.0 & {\bf 1.7} & 5.9 & 21.8 & 16.2
\end{tabular}
\end{table} }

{\tiny
\begin{table}[]
\caption{OPsym instances: bounds and gaps}
\label{Sym}
\begin{tabular}{rrrrrrrrr c rrrr}
\hline
\multicolumn{2}{c}{instance} & &  \multicolumn{5}{c}{lower bounds} & \multicolumn{5}{c}{Gap ($\%$)}  \\
\cline{1-2}   \cline{4-8} \cline{10-14} \\
$n$ & $d (\%)$  & $UB$ & $GL$ & $\widetilde{RLT_1}$ & $ \widetilde{RLT_2}$ & $VS_1$ & $VS_2$ & & $GL$ & $ \widetilde{RLT_1}$ & $ \widetilde{RLT_2}$  & $VS_1$ & $VS_2$ \\
\hline
6 & 100 & 258.4 & 147.5 & 253.5 & 254.8 & 246.4 & {\bf 257.4} &  & 42.9 & 1.9 & 1.4 & 4.6 & {\bf 0.4} \\
7 & 100 & 326.8 & 167 & 310.4 & 317.6 & 294.6 & {\bf 325.9} &  & 48.9 & 5.0 & 2.8 & 9.9 & {\bf 0.3} \\
8 & 100 & 438.5 & 203.9 & 399.7 & 414.5 & 361.9 & {\bf 434.4} &  & 53.5 & 8.8 & 5.5 & 17.5 & {\bf 0.9} \\
9 & 100 & 534.9 & 232.2 & 466.3 & 496.7 & 418.8 & {\bf 531.8} &  & 56.6 & 12.8 & 7.1 & 21.7 & {\bf 0.6} \\
10 & 100 & 653.9 & 240.3 & 529.9 & 584.4 & 473.2 & {\bf 638.8} &  & 63.3 & 19.0 & 10.6 & 27.6 & {\bf 2.3} \\
11 & 100 & 785.9 & 250.4 & 584.3 & 668.3 & 513.4 & {\bf 738.7} &  & 68.1 & 25.7 & 15.0 & 34.7 & {\bf 6.0} \\
12 & 100 & 918.5 & 256.1 & 646.8 & 755.3 & 565.4 & {\bf 828.4} &  & 72.1 & 29.6 & 17.8 & 38.4 & {\bf 9.8} \\
13 & 100 & 1067.1 & 282.4 & 706.7 & 848.9 & 614.4 & {\bf 900.3} &  & 73.5 & 33.8 & 20.4 & 42.4 & {\bf 15.6} \\
14 & 100 & 1249.8 & 300 & 799.1 & 968.8 & 693.7 & {\bf 1006.1} &  & 76.0 & 36.1 & 22.5 & 44.5 & {\bf 19.5} \\
15 & 100 & 1390.2 & 287.8 & 815.8 & 1032.1 & 694.6 & {\bf 1045.9} &  & 79.3 & 41.3 & 25.8 & 50.0 & {\bf 24.8} \\
16 & 100 & 1629.3 & 299.2 & 886.9 & {\bf 1270.9} & 761.7 & 1129.8 &  & 81.6 & 45.6 & {\bf 22.0} & 53.2 & 30.7 \\
17 & 100 & 1823.8 & 324 & 961.9 & {\bf 1258.7} & 826.1 & 1223.8 &  & 82.2 & 47.3 & {\bf 31.0} & 54.7 & 32.9 \\
18 & 100 & 2981 & 344.4 & 1467.0 & {\bf 1998.0} & 880.8 & 1298.2 &  & 88.4 & 50.8 & {\bf 33.0} & 70.5 & 56.5 \\
20 & 100 & 2572.7 & 339 & 1132.8 & {\bf 1608.9} & 967.3 & 1441.8 &  & 86.8 & 56.0 & {\bf 37.5} & 62.4 & 44.0 \\
30 & 100 & 6015.9 & 401.3 & 1664.1 & {\bf 2843.1} & 1396.8 & 2123.3 &  & 93.3 & 72.3 & {\bf 52.7} & 76.8 & 64.7 \\
\end{tabular}
\end{table} }

\section{Conclusion} \label{conclusion}

This paper introduces a  hierarchy  of lower  bounds for the quadratic minimum spanning tree problem.
Our bounds exploit an extended formulation for the MSTP from  \cite{MARTIN1991119} and the linear, exact formulation for the QMSTP from \cite{AssadXu}.
We prove that our simplest relaxation $VS_0$ is equivalent to the linearization based relaxation derived in Section \ref{sect:linBasedBnd}, see Theorem \ref{prop:equivalent}.
To improve the relaxation $VS_0$  we add facet defining inequalities of the Boolean Quadric Polytope.
The resulting relaxations  $VS_1$ and  $VS_2$ are presented in  Table \ref{VS_relaxations}.
Our relaxations have a polynomial number of constraints and can be solved by a cutting plane algorithm.

On the other hand, all relaxations in the literature have an exponential number of constraints and most of them  belong to the RLT type of bounds, see Section \ref{sect:overwierBNDS}.
The fact that our relaxations differ in both mentioned aspects from the other  relaxations,  enables us to efficiently compute bounds that are not dominated by the RLT type of bounds. \\

\medskip
\noindent
\textbf{Acknowledgements.} We would like to thank Frank de Meijer for an insightful discussions on linearization based bounds.
We also thank Renate van der Knaap for programming heuristics algorithms and computing upper bounds.

\bibliographystyle{plainnat}
\bibliography{bibQMST}

\end{document}